\documentclass{amsart}

\usepackage[english]{babel}

\usepackage{geometry}
\usepackage{amsmath}
\usepackage{graphicx}
\usepackage[colorlinks=true, allcolors=blue]{hyperref}

\usepackage[utf8]{inputenc}
\usepackage{tikz-cd}
\usepackage[utf8]{inputenc}
\usepackage[T1]{fontenc}
\usepackage{amssymb}
\usepackage[all]{xy}
\usepackage{graphicx}
\usepackage{amsmath}
\usepackage{colonequals}
\usepackage{multicol}
\usepackage{amsthm}
\usepackage{hyperref}
\newtheorem{thm}{Theorem}
\newtheorem{lem}[thm]{Lemma}
\newtheorem{defi}[thm]{Definition}
\newtheorem{prop}[thm]{Proposition}
\newtheorem{cor}[thm]{Corollary}
\newtheorem{ex}[thm]{Example}
\newtheorem{rmk}[thm]{Remark}

\newcommand{\bZ}{\mathbb{Z}}
\newcommand{\bQ}{\mathbb{Q}}

\title{Geometric realizations of non-symplectic involutions on the Hilbert square of a K3 surface}
\author{Ana Quedo}
\address{Dipartimento di Matematica - Università di Bologna Piazza di Porta S. Donato, 5, 40126 Bologna BO}
\email{ana.martinsquedo@unibo.it}


\begin{document}
\maketitle

\begin{abstract}
We give new examples of geometric constructions of non-natural non-symplectic involutions of IHS manifolds whose existence is guaranteed by previous results of Bossière-Cattaneo-Nieper-Wiesskirchen-Sarti in \cite{BCNS} and Bossiére-Camere-Sarti in \cite{BCS}.

\end{abstract}

\section{Introduction}
The Beauville-Bogomolov Decomposition Theorem asserts that a compact Kähler manifold with trivial first Chern class can be constructed exclusively using three distinct types of building blocks: complex tori, Calabi-Yau manifolds, and irreducible holomorphic symplectic (IHS) manifolds.

Calabi-Yau and IHS manifolds represent distinct approaches to generalizing K3 surfaces in higher dimensions. An IHS manifold is a smooth, compact, simply connected, complex Kähler manifold with a holomorphic 2-form everywhere that is non-degenerate, and unique up to scalar multiplication. 

In the present paper, we focus on automorphisms of a specific type of IHS fourfold, namely the Hilbert scheme of two points on a K3 surface. Given a complex projective K3 surface $S$, we define the Hilbert scheme of 2 points on a K3 surface $S$, or simply the \textit{Hilbert square} $S^{[2]}$ as the Hilbert scheme of subschemes of dimension 0 and length 2 on $S$. There are two main reasons why this particular IHS is interesting. First, they are one of the two known examples of IHS fourfolds (up to deformation families). Also, we can apply some aspects of our knowledge of automorphisms of K3 surfaces to study the automorphisms of these fourfolds. 

 Many papers have worked towards the goal of studying and classifying automorphisms of the Hilbert square of a K3 surface for example \cite{BCS}, \cite{BCMS}, \cite{BCNS}.
 However, in general, these classification results do not provide a geometric description of these automorphisms, that is given an automorphism $\sigma \colon S^{[2]} \to S^{[2]}$ and a point $p \in S^{[2]}$, we are not able to calculate $\sigma(p)$. Examples of geometric realizations of automorphisms of hyperkähler manifolds have been provided by several authors, including Beauville \cite{Beauville}, the famous Beauville involutions (see Definition \ref{beauville}), O'Grady \cite{OG} and, more recently, Beri and Cattaneo \cite{bericat} and Ohashi and Wandel \cite{OW}.

In this paper, we give new geometrical constructions of non-natural non-symplectic involutions of the Hilbert square of a K3 surface following the same philosophy: involutions whose construction is somehow derived from K3 surfaces, even though they are non-natural. In this way, we expect to create a new set of examples that are easier to understand geometrically.

In section 3, we use the set of examples of K3 surfaces $S_n$ constructed by Paiva and the author in \cite[Example 1]{paivaquedo} for that. This set of surfaces was first studied because of its relation with Gizatullin's problem, and it has remarkable properties, for example, $S_n$ admits a square two ample divisor $W$ and $\mathbb{Z}_2 \ast \mathbb{Z}_2 \subseteq Aut(S_n)$.

In fact, for $S$ an algebraic K3 surface such that $Pic(S) = \bZ L$ with $L^2 = 2t$ with $t$ satisfying certain numerical conditions, \cite[Theorem 5.5]{BCNS} (Theorem \ref{Sarti}) predicts the existence of a non-symplectic involution $\sigma$ on the Hilbert square $S^{[2]}$.

We show that $S_n^{[2]}$ admits Beauville involutions $i_1$ and $i_2$,
then consider $\kappa_1:=i_2i_1i_2 $ and $\kappa_2:=i_1i_2i_1$ and show that under certain conditions the non-symplectic involution $\sigma$ on $S^{[2]}$ can be deformed to both $\kappa_1$ and $\kappa_2$. 
More precisely, we prove 

 \begin{thm} \label{t1}Consider $(S, L)$ a K3 surface with polarization $L$ of square $2t$, $t\geq 2$ and $Pic(S)= \bZ L$.  
 Then, 
 $S^{[2]}$ admits a unique non-symplectic involution $\sigma$ such that
 $(S^{[2]}, \sigma)$ can be deformed into $(S_n^{[2]}, \kappa_i)$, j=1,2, in a way that every element of the deformation family is of the form $(\Sigma^{[2]}, j) $, where  $\Sigma$ is $2t$-polarized K3 surface and $j$ an involution of $\Sigma^{[2]}$ if and only if $t$ can be written as $(64n^{2}-7)^{2}+1$ or  $2^{12} \cdot 5^{2}k^{4} - 2^{5} \cdot 7k^{2} +2$.
 \end{thm}

In section 4, we accomplish another main goal of this paper: we use again the set of surfaces $S_n$ to produce a new geometric realization of a non-natural involution of the Hilbert scheme of two points. Bossière, Camere, and Sarti classified all possible non-symplectic automorphisms of an IHS manifold of type $K3^{[2]}$ in \cite{BCS} by studying p-elementary-lattices. However, there are several cases where examples are still unknown.

The second main theorem in this paper is the following:

\begin{thm} \label{t2}
    There is a geometric example of a non-natural involution on the Hilbert scheme of 2 points of a K3 with the invariant lattice of $<-2> \oplus <2>$.
\end{thm}

Although examples for involutions on an IHS of type $K3^{[2]}$ with this invariant lattice are already known in the literature, a construction done by Ohashi and Wandel in \cite{OW} and the natural involution $\varphi^{[2]}$ on $S^{[2]}$ induced by an involution $\varphi$ on the original K3 $S$, our example has an explicit geometric construction, even though it is a non-natural involution. In fact, in our construction, we consider the natural involution induced on $S_n^{[2]}$ by the 2-square divisor of the surface $S_n$ and conjugate it with a Beauville involution.  

The paper is structured as follows. Section 2 gathers the necessary background to the main results. It covers basic definitions of Pell equations, Lattices, Isometries, and Hyperkähler manifolds and their automorphisms. The new results of this paper are presented in sections 3 and 4.  In section 3, we present geometric constructions of automorphisms through the deformation of non-symplectic involutions as stated in Theorem \ref{t1}.
In section 4, we present a new example of geometric non-natural involution with a rank 2 invariant lattice as stated in Theorem \ref{t2}.

\textbf{Acknowlegments} I would like to thank Alessandra Sarti and Pietro Beri for the many discussions that culminated in this paper, for the careful reading of its first drafts, and particularly to Alessandra for introducing me to this new subject, and for all the support during its redaction. Furthermore, gratitude is expressed to CAPES (Coordenação de Aperfeiçoamento de Pessoal de Nível Superior) for the financial support received during my stay in Poitiers, grant number 88887.511749/2020-00,
which made this article possible. During the process of finalization of this paper, I was also supported by thr
funding from the European Union - NextGenerationEU
under the National Recovery and Resilience Plan (PNRR) - Mission 4 Education
and research - Component 2 From research to business - Investment 1.1, Prin
2022 "Geometry of algebraic structures: moduli, invariants, deformations", DD N.
104, 2/2/2022, proposal code 2022BTA242 - CUP J53D23003720006.

\section{Preliminaries}
\subsection{Pell equations}Given $d$ a positive non-square number and $m \in \mathbb{Z}\backslash \{0\}$, a diophantine equation of the form

\begin{equation}\label{eq:1}
    P_d(m) \colon x^{2}- dy^{2}= m
    \end{equation}

\noindent is called a \textit{generalized Pell equation}. When $m=1$, we called it the \textit{standard Pell equation}.
We say that a solution $(x,y) \in  \bZ^{2}$ of \ref{eq:1} is a \textit{positive solution} if both $x>0$ and $y>0$ and the \textit{minimal positive solution} $(x_1,y_1)$ is the one with minimal positive $x_1$. The next theorem is due to Lagrange \cite{Lagrange}, but a simple proof of this assertion can be found in \cite[Theorem 2.3]{Conrad}.

\begin{thm}
    For every positive non-square integer $d$, the standard Pell equation $P_d(1):x^{2}-dy^{2}=1$ has a positive solution.
\end{thm}

\begin{cor}
    If a generalized Pell equation $P_d(m)=x^{2}-dy^{2}=m$ has a positive solution, then it has infinitely many solutions.
\end{cor}
\begin{proof}
  Given $(a,b)$ a solution of standard Pell equation $P_d(1)$ and a solution of $(u,v)$ of $P_d(m)$, we can verify that $(ua+vbd, va+ub) $ is another solution for $P_d(m)$. By induction, we have infinitely many solutions.
\end{proof}



\begin{prop}\label{pell2}
    Let $d_n= 4(n^{2}-2)$. If $n$ is a multiple of four, the Pell equation $P_{d_n}(-8): x^{2}- d_ny^{2}=-8$ does not have solutions.
\end{prop}
\begin{proof}
    Suppose $n=4n'$, then we can rewrite the equation as $$P_{d_n}(-8): x^{2}-8(8n'^{2}-1)y^{2}=-8.$$ For any solution $(x,y)$ of $P_{d_n}(-8)$, it is clear that $x^{2}$ is a multiple of 8, and since it is a square it is also a multiple of 16 and we can write $x^{2}=16x'^{2}$. Thus, the equation $x^{2}-8(8n'^{2}-1)y^{2}=-8$ has integer solutions if and only if 
    \begin{equation}
    \label{eq}
            2x'^{2}-(8n'^{2}-1)y^{2}=-1
    \end{equation}
     has integer solutions.
     Let's see that the last equation does not have solutions modulo 8. We recall that there are 3 possibilities for a square number modulo 8:
    \begin{itemize}
        \item $a\equiv 2 \text{ or } 6 \mod{8} \Longrightarrow a^{2} \equiv  4 \mod{8}$.
        \item $a\equiv 0 \text{ or } 4 \mod{8} \Longrightarrow a^{2} \equiv  0 \mod{8}$.
        \item $a$ is odd  $\Longrightarrow a^{2} \equiv  1 \mod{8}$.
    \end{itemize}
If $x'$ is an even number, the equation (\ref{eq}) modulo 8 is reduced to $y^{2}=7 \mod{8}$, which obviously does not have solutions.
Similarly if $x'$ is an odd number, (\ref{eq}) modulo 8 is reduced to $y^{2}=5 \mod{8}$.
We then prove our result.
\end{proof}

\subsection{Lattices}
\begin{defi}[Lattice] A \emph{lattice} $L$ is a free $\mathbb{Z}$-module with finite rank together with a symmetric bilinear form $$(\ ,\ ):L\times L\longrightarrow\mathbb{Z}.$$

\begin{enumerate}
 \item A lattice is called \emph{even} if $(x,x)\in 2\mathbb{Z}$ for all $x\in L$. 
 \item Since $L$ is a free $\bZ$ module, we can choose a basis $\{x_1, \dots, x_n\}$ and represent the symmetric bilinear form by the associated Gram matrix, i.e,
 the matrix $G$ with entries $G_{ij}=(x_i, x_j)$. The determinant of $G$  is called the \emph{discriminant} of $L$, denoted by $disc(L)$, and it is independent of the choice of basis. If $disc(L)\neq0$, $L$ is called \emph{non-degenerate}.
  \item  The bilinear form $(\ ,\ )$ can be extended to a symmetric bilinear form on the real vector space $L\otimes\mathbb{R}$. Assuming the lattice $L$ is non-degenerate, the bilinear form on $L \otimes \mathbb{R}$ is equivalent
to one represented by a diagonal matrix whose nontrivial entries are either 1 or $- 1$. The \emph{signature} of $L$ is the pair $(n_+,n_-)$, where $n_{\pm}$ is the number of $\pm 1$ on the diagonal. 
\item The lattice is called \emph{definite} if either $n_{+} = 0$ or $n_{-} = 0$, otherwise is called \emph{indefinite}.

  \end{enumerate}
  \end{defi}

 \begin{defi}[Discriminant group]
We can define an injection of finite index
\begin{align*}
L &\hookrightarrow L^{\vee}:= Hom_{\bZ}(L, \bZ) \\
x &\mapsto (x, \cdot)
\end{align*}

\noindent The lattice $L^{\vee}$ is called the \emph{dual lattice} of $L$. We set $L\otimes\mathbb{Q}:=L_{\mathbb{Q}} $ and denote by $L_{\bQ}^{\vee}$ the dual vector space, we notice that $L^{\vee} \hookrightarrow L^{\vee}_{\bQ} \simeq L_{\bQ}$.
So, we have 
  $$L^{\vee}=\{x\in L_{\bQ} \text{ | }(x,y)\in\mathbb{Z}\ \text{for any }y\in L\}.$$
  
The quotient $A(L):=L^{\vee}/L$ is called the \emph{discriminant group}, which is a group of order $|disc(L)|$.
\begin{enumerate}
\item We say $L$ is \emph{unimodular} if $L^{\vee} \simeq L$, or equivalently if $disc(L)=\pm 1$.
\item We denote by $l(A(L))$ the minimal number of generators of $A(L)$ (i.e. the length of $A(L)$).
\item  A lattice $L$ is called \emph{$p$-elementary}, for a prime $p$, if $A(L)\cong (\bZ\slash p\bZ)^a$, for some positive integer $a= l(A(L))$, 
\end{enumerate}
\end{defi}

We see some examples of lattices.
\begin{ex} $\langle k \rangle$ denotes the lattice of rank 1 such that $(x, x)= k$ for any generator 
$x$ of  $\langle k \rangle$.
\end{ex}

\begin{ex} \label{u}
The hyperbolic plane is the lattice
\begin{center}
$U:=\begin{pmatrix}
0 & 1 \\
1 &0
\end{pmatrix}$
\end{center}

\noindent i.e. $U \simeq \bZ^2 = \bZ e \oplus \bZ f$ with $(e, e)= (f,f)=0$ and $(e,f) = 1$.
Clearly, $disc (U) = -1.$
\end{ex}

\begin{ex}

The $E_8$-lattice is given by the intersection matrix
\begin{center}
$E_8 := \begin{pmatrix}
2 & -1 &    &     &         &     &   &     \\
-1 & 2 & -1 &    &         &     &   &     \\
    &-1& 2  &-1 & -1     &     &    &      \\
    &   &-1 & 2 &       &     &     &    \\
    &   &-1  &  & 2  &-1  &     &          \\
     &   &   &  & -1 & 2 & -1   &   \\
     &   &    & &    &  -1& 2    &-1\\
      &  &    &  &     &   &  -1 &2 
\end{pmatrix}$
\end{center}
\noindent and is, therefore, even, unimodular, positive definite of rank eight with
$disc (E_8) = 1$.
\end{ex}

We end this subsection with a lemma.
\begin{lem} \label{latticecondition} \cite[Lemma 9]{paivaquedo}Consider a lattice  $L=\mathbb{Z}h_1 \oplus\mathbb{Z}h_2$ with bilinear form given by the matrix $Q$
\begin{equation} \label{intermatrix}
Q =\begin{pmatrix}

 4& b \\ 
b & 2c
\end{pmatrix}
\end{equation}
Let $r := - discr(L)$ and $k\neq 0$ an integer number. Then, the following assertions hold:
\begin{enumerate}
\item {For any $v\in L$ we have that $4v^2=(v\cdot h_1)^2-rn^2$, for some integer $n$.}
    \item There are vectors $v \in L$ such that $v^{2}=0$ if and only if r is a square.
    \item If the Generalized Pell equation $x^2 - ry^2 = 8k$ has no integer solutions, then there are no vectors $v \in  L$ such that $v^2 = 2k$. 
\end{enumerate}
\end{lem}

\subsection{Isometries}
We define isometries of lattices.
\begin{defi}[Isometry]  An isometry between lattices $(L_1, (,)_{L_1})$ and $(L_2, (,)_{L_2})$ is an isomorphism of $\bZ$-modules 
$\varphi \colon L_1 \to L_2$ such that  $(\varphi(x), \varphi(y))_{L_2} = (x, y)_{L_1}$ for every $x, y  \in L_1$. We denote by $O(L)$ the group of 
isometries of a lattice $L$ into itself and call it the \emph{orthogonal group of L}.
\end{defi}

\begin{defi}
\cite[Definition 1.8]{Morrison} An embedding of lattices is an injective morphism $i \colon M \to L$  of $\bZ$-modules such that
$(x,y)_{M}=(i(x), i(y))_{L}$. 
\begin{enumerate}
\item An embedding $M \hookrightarrow L$ of lattices is \emph{primitive} if $L/M$ is free. 
\item An element $x \in L$ is {primitive} if the inclusion $i \colon \langle x \rangle  \hookrightarrow L $ is primitive.  In other words, there is no  $x' \in L $ such that $x = kx'$ for some integer $k \geq 2$.
\item Two primitive embeddings $\varphi_1 \colon M \to L_1$, $ \varphi_2 \colon M\to L_2$ are \emph{isomorphic} if there is an isometry 
$f \colon L_1 \to L_2$ which induces the identity map on M, i.e., the following diagram is commutative.
\end{enumerate}
\begin{center}
\begin{tikzcd}
M \arrow[r, "\varphi_1", hook] \arrow[d, "Id"] & L_1 \arrow[d, "f"] \\
M \arrow[r, "\varphi_2", hook]                 & L_2               
\end{tikzcd}
\end{center}

\end{defi}


We recall the definition of reflection by an element $D$ of the lattice.

\begin{defi}[Reflection of an element D]\label{reflection} Consider $L$ a non-degenerate lattice. Given $D \in L$, we define the reflection 
$$R_D \colon  l \mapsto l - 2\dfrac{(l, D)}{(D, D)}D$$

We call $-R_D$ the anti-reflection in $D$.
\end{defi}

\subsection{Hyperkähler Manifolds}

\begin{defi}[Hyperkähler Manifold]
A Hyperkähler manifold is a compact complex simply connected Kähler manifold $X$ such that $H^0(X, \Omega_X^2) = \mathbb{C}\omega$, where $\omega$ is an everywhere non-degenerate holomorphic 2-form on $X$.
\end{defi}

K3 surfaces are the easiest examples of hyperkähler manifolds. We recall some properties that will be useful for us.

\begin{defi}

The \emph{Positive cone} $P(S)$ of a K3 surface $S$ is the connected component of $$\{x\in Pic(S)\otimes_{\mathbb{Z}} \mathbb{R}| (x,x)>0\}$$ which contains ample classes.
\end{defi}

A ($-2$)-curve on $S$ is an irreducible curve $C$ with $C^2=-2$. In particular,  by adjunction, ($-2$)-curves are rational, \emph{i.e.}, $C\simeq \mathbb{P}^1$, and since they are the only curves of a K3 surfaces with negative self-intersection. We have the following description for an ample cone of a K3 surface.

\begin{prop}\cite[Corollary 1.7, Chapter 8]{Huybrechts}

\label{amplecriterio}For a K3 surface $S$, the ample cone is given by
$$Amp(S)= \{ x \in P(S)| x \cdot C > 0 \text{ for all ($-2$)-curves $C$} \}. $$\end{prop}

 The following result is a combination of results proven in \cite{SaintDonat}, but the following formulation can be found in \cite[Theorem 5]{Mori}.

\begin{thm}\label{Mori}
     Let $S$ be a K3 surface and $H$ a nef line bundle with $H^2\geq 4$. Then, $H$ is very ample if and only if it satisfies the following 3 conditions:
     \begin{itemize}
      
      \item There is no irreducible curve $E$ such that $E^{2}=0$ and
      $H \cdot E\in\{1,2\}$.
      \item There is no irreducible curve $E$ such that $E^{2}=2$, and $H \sim 2E$.
      \item There is no irreducible curve $E$ such that $E^{2}=-2$ and
      $E \cdot H=0$.
         
     \end{itemize}
\end{thm}

In higher dimensions, Beauville proved in \cite[Theoreme 3]{beauvillefr}, that given $S$ a projective K3 surface, the corresponding 
Hilbert scheme of $n$ points on $S$, denoted here by $S^{[n]}$, is an example of a projective hyperkähler manifold of dimension $2n$. 
We recall now some properties of this class of varieties. For that, we base ourselves on the exposition
done by Beauville in \cite[Section 6]{beauvillefr}. 

\begin{ex} Let $S$ be a K3 surface. The Hilbert scheme of $n$ points $S^{[n]}$ parametrizes the 0-dimensional subschemes of lenght $n$ of $S$. We denote by $S^{(n)}$ the quotient of $S^{n}$ by the symmetric group $\mathcal{S}_n$, by $\pi \colon S^n \to S^{(n)} $ the quotient map and notice that $S^{(n)}$ can be thought of as the variety of $0$-effective cycles of degree $n$ of $S$. Moreover, we denote by $\epsilon \colon S^{[n]} \to S^{(n)}$ the Hilbert-Chow morphism, which is the map that associates to a finite subscheme the corresponding 0-cycle. Finally, considering the set $\Delta_{ij}= \{(x_1, \dots,x_n) \in S^n \text{ | } x_i=x_j \} \subset S^n$, $\Delta= \cup_{i\neq j}\Delta_{ij}$ and
$D= \pi(\Delta) \subset S^{(n)}$, we can state some properties associated with this example.
\begin{itemize} \item $S^{[n]}$ is smooth and the singular locus of $S^{(n)}$ is $D$.
\item The morphism $\epsilon \colon S^{[n]} \to S^{(n)}$ is birational and a desingularization of $S^{(n)}$.
\item $\epsilon^{-1}(D)=E$ an irreducible divisor.
\end{itemize}
\end{ex}

 If $X$ is deformation equivalent to $S^{[n]}$ for some K3 surface $S$, we say $X$ is of type $K3^{[n]}$. There are few known examples of hyperkähler manifolds up to deformation equivalence.
 Besides K3 surfaces and $K3^{[n]}$, the other known deformation classes of hyperkähler manifolds are generalized Kummer varieties and the constructions given by O'Grady in \cite{OG} and \cite{OG2}.

We see now that for a projective hyperkähler manifold $X$, the group $H^2(X, \bZ)$ carries a natural lattice structure.

\begin{prop} Let $X$ be a projective hyperkähler manifold with non-degenerate holomorphic 2-form $\omega$, there is a natural integral symmetric bilinear form on $H^{2}(X,\bZ)$
such that the pair $(H^2(X, \bZ), \langle, \rangle ) $ is a lattice of signature $(3, b_2(X)- 3)$. The bilinear $ \langle, \rangle $ is called the Beauville–Bogomolov form, and it
 satisfies the following properties
\begin{itemize}
\item $\langle \omega, \omega \rangle = 0, \langle \omega,\bar{\omega} \rangle > 0$,
\item $H^{1,1}(X) = (H^{2,0}(X)  \oplus H^{0,2}(X))^{\perp} \subset H^2(X, \mathbb{C})$,
\item $\langle x, x \rangle > 0$ for every Kähler class $x$ on $X$.
\end{itemize}

\end{prop}

For a K3 surface $S$, the Beauville–Bogomolov form coincides with the intersection product of the surface, for the Hilbert scheme of $n$ points, we also have a nice description of this form.

\begin{prop}\cite[Proposition 6 and Remarque]{beauvillefr} \label{6}
There exists an injective homomorphism $j\colon H^2(S, \bZ) \to H^2(S^{[n]}, \bZ)$ preserving the Beauville–Bogomolov form such that
$H^2(S^{[n]}, \bZ) = j(H^2(S, \bZ)) \oplus \bZ\delta$ and $2\delta =E$ and $\langle \delta, \delta \rangle=2 -2n$.
Moreover, $NS(S^{[n]}) = j(NS(S)) \oplus \bZ\delta$.

\end{prop}
\begin{rmk}
Let $S$ be a K3 surface and $S^{[2]}$ the corresponding Hilbert scheme of 2 points.
To simplify our notation, from now on, we denote by $\langle,\rangle$ both the intersection form on $S$ and the Beauville-Bogomolov form on $S^{[2]}$.
Also, given $H \in NS(S)$, we denote by $[H]$ the image $j(H) \in NS(S^{[2]})$. 
\end{rmk}

\begin{cor}
Let $S^{[n]}$ be the Hilbert scheme of $n$ points of a K3 surface $S$. The integral cohomology $H^2(S^{[n]}, \bZ)$ endowed the Beauville–Bogomolov form is abstractly isomorphic to the lattice
$$H^2(S^{[n]}, \bZ) \simeq E_8(-1)^{\oplus 2} U^{\oplus 3} \oplus \langle 2 -2n \rangle.$$
\end{cor}

\subsection{Automorphisms}

Some results about automorphisms of K3 surfaces can be generalized for hyperkähler manifolds. For example, the Global Torelli Theorem still holds for hyperkähler manifolds, as proven by Verbitsky in \cite{misha}.

From now on, we focus on the involutions of the Hilbert scheme of $2$ points of a K3 surface, also known as \emph{the Hilbert square of a K3}. As one can guess, an involution of a K3 surface induces naturally an involution of the Hilbert scheme of $2$ points. This is formalized in our next definition.

\begin{defi}[Natural involution] 
Let $i \colon S \to S$ be an involution of a K3 surface $S$ and $Z \in S^{[2]}$. Since $Z$ is a  0-dimensional subscheme of lenght 2 of $S$, we can define the natural involution $i^{[2]}$ on the Hilbert square $S^{[2]}$ in the following way:
 
 \begin{align*} i^{[2]} \colon S^{[2]} &\to S^{[2]} \\
 Z &\mapsto i(Z)
\end{align*}

\end{defi}

An important property of natural automorphisms of Hilbert schemes of $n$ point was proven by Bossière and Sarti. Here we state the result for natural involutions on the Hilbert square of a K3 surface.

\begin{thm}\cite[Theorem 1]{BS}
\label{BS}
Let $S$ be a K3 surface. An involution $f$ of $S^{[2]}$ is
natural if and only if it leaves globally invariant the exceptional divisor of $\epsilon \colon S^{[n]} \to S^{(n)}$, or equivalently $f^{*}(\delta)=\delta.$
\end{thm}

In general, it is not easy to give geometrical constructions of non-natural involutions of $S^{[2]}$. We present now one of the few examples that are known, a construction due to Beauville \cite{Beauville}.

Consider $S \subset \mathbb{P}^{3}$ a smooth quartic surface. We have that a general point $Z$ of the Hilbert scheme $S^{[2]}$ corresponds to a set of two distinct points in $S$. These two points define a line $l(Z) \subset \mathbb{P}^{3}$. If this line is not contained in $S$, $l(Z)$ meets $S$ generically  in 4 four distinct points, that is
$$ l(Z) \cap S= Z \cup Z'$$
where $Z' \in S^{[2]}$.
By associating $Z$ to $Z'$, we define a birational map $$i \colon S^{[2]} \dashrightarrow S^{[2]} .$$
This map is an isomorphism if and only if $S$ has no lines, as proved in \cite[Proposition 11]{Beauville}.
Furthermore, assuming that $S$ has no lines and $H$ is the very ample divisor of square 4 that defines the embedding $S \subset \mathbb{P}^{3}$, O' Grady showed in \cite[Section 4.1.2]{OG} that $i^{*}$ acts on $H^2(S^{[2]},\mathbb{Z})$ as the reflection in the ample divisor class $D= [H]- \delta$.
Therefore, we can define:

\begin{defi}[Beauville Involution]
 \label{beauville}Let $S$ be a smooth quartic surface.  A Beauville involution on $S^{[2]}$ is a non-symplectic involution $i$ whose invariant lattice has rank 1 and is generated by an ample divisor $D$ of the form $D= [H] - \delta$, where $H$ is an ample divisor of square 4 in $S$.   
\end{defi}

In the next theorem, we describe an association between ample square 2 divisors and nonsymplectic involution for hyperkähler manifolds of type $K3^{[2]}$. But before, we give an important definition.

\begin{defi}[Invariant Lattice] Given $g$ an automorphism of a hyperkähler manifold $X$, we can define its invariant lattice $$H^{2}(X, \bZ)^g=\{ x \in H^{2}(X, \bZ) \text{  |  }g^{*}x=x    \}.$$

\end{defi}
\begin{thm}\cite[Theorem 3.1]{BCMS} \label{bcms}
\begin{enumerate}
 \item Let $(X, D)$ be an ample $\langle 2 \rangle$-polarized hyperkähler manifold of type $K3^{[2]}$. Then
$X$ admits a nonsymplectic involution $\sigma$ whose action on $H^2(X, \bZ)$ is the
anti-reflection (see Definition \ref{reflection})  in the class of $D$ in $H^2(X, \bZ)$.
\item Conversely, let $X$ a be hyperkähler manifold of type $K3^{[2]}$ with a nonsymplectic
involution $\sigma$ whose invariant lattice $H^{2}(X, \bZ)^{\sigma}$ has rank 1. Then $H^{2}(X, \bZ)^{\sigma}$ is generated by the class of an ample divisor $D$ of square 2 and $\sigma$ acts on $H^2(X, \bZ)$ as the anti-reflection (see Definition \ref{reflection} in the class of $D$.
\end{enumerate}
\end{thm}

We finish this section with an important result of Bossière-Cattaneo-Nieper-Wiesskirchen-Sarti \cite{BCNS} about the existence of automorphisms on the Hilbert square of a K3 surface.

\begin{thm} \label{Sarti} \cite[Theorem 5.5]{BCNS}
Let $S$ be an algebraic K3 surface such that $Pic(S) = \bZ L$ with
$L^2 = 2t$, $t \geq 2$. Then $S^{[2]}$ admits a non-trivial automorphism if and only if one of
the following equivalent conditions is satisfied:

\begin{enumerate}
\item  $t$ is not a square, the generalized Pell equation $P_{4t}(5):x^{2}- 4ty^{2}=5$ has no solution and the generalized Pell
equation $P_{t}(-1):x^{2}-ty^{2}=-1$ has a solution;
\item There exists an ample class $D \in NS(S^{[2]})$ such that $D^2 = 2$.
Moreover, in this case, the class $D$ is unique, the automorphism is unique and
it is a non-symplectic involution whose action on $H^2(S^{[2]}, \bZ)$ is the anti-reflection in the
span of $D$.
\end{enumerate}
\end{thm}

\begin{rmk}\label{obs1}
We denote by $(a,b)$ the minimal positive solution of Pell's equation $P_{t}(-1)$ and by $[L]$ the class of $L$ in
$ NS(S^{[2]})= NS(S) \oplus \mathbb{Z} \delta $. In fact, from the proof of the above theorem, we have that the divisor $D=(b[L]-a\delta)$.
\end{rmk}

\subsection{Deformation equivalence}
In this section, we discuss the basic definitions of deformation equivalence, focusing on hyperkähler manifolds, specifically the Hilbert square of K3 surfaces.
We base ourselves on the exposition of \cite[Chapter 7, Section 1]{Huybrechts} and \cite{Joumaah}.

\begin{defi} [Deformation equivalence] 
Two compact complex manifolds $X_1$ and $X_2$ are said to be deformation equivalent if there
exists a smooth proper holomorphic morphism $\mathcal{X} \to B$ over a (possibly singular) connected base $B$, points $t_1, t_2 \in B$ and isomorphisms $\mathcal{X}_{t_1} \simeq X_1$ and $\mathcal{X}_{t_2} \simeq X_2$.

\end{defi}

Let $\pi \colon \mathcal{X} \to T $ be a smooth proper family of hyperkähler manifolds over a connected smooth analytic space $T$.
We define a \emph{family of non-symplectic involutions over $\mathcal{X}$} as the pair $(\mathcal{X}, I)$, where $I\colon \mathcal{X} \to \mathcal{X}$ is a holomorphic involution such that $ \pi \circ I= \pi$ and the restriction to a fiber $I_t: \mathcal{X}_t \to \mathcal{X}_t$ is a non-symplectic involution $\forall t \in T$.

\begin{defi}[Deformation equivalence of involutions] \label{definv} Let $X_1$ and $X_2$ be hyperkähler manifolds and $i_1$ and $i_2$ non-symplectic involutions on $X_1$ and $X_2$, respectively. The pairs $(X_1,i_1)$ and $(X_2,i_2)$ are deformation equivalent, if there exists a family $(\pi, I)$ : $\pi \colon \mathcal{X} \to T $ of non-symplectic involutions and points $t_j \in T$ with $(X_{t_j}, I_{t_j}) \simeq (X_j,i_j)$ for
$j = 1, 2$.
\end{defi}

We recall some results and notations established in \cite[Section 4]{BCMS}.
We fix a primitive embedding $j \colon \langle 2 \rangle  \to U^{\oplus 3} \oplus E_8(-1)^{\oplus 2} \oplus \langle-2 \rangle := L$ . Such an embedding is unique up to an isometry of  $U^{\oplus 3} \oplus E_8(-1)^{\oplus 2} \oplus \langle-2 \rangle$ (see \cite[Proposition 8.2]{BCS}). We denote by $\rho \in O(L)$ the anti-reflection in the ray $j(\langle 2 \rangle)$.
By Bossière-Camere-Sarti \cite[Theorem 4.5,Theorem 5.6]{complexballs}, there is a quasi-projective coarse moduli space $\mathcal{M}_{\rho}^{\langle 2\rangle}$ parametrizing triples $(X, i, l)$, where $X$ is a hyperkähler manifold of type $K3^{[2]}$, $l \colon \langle 2\rangle \hookrightarrow NS(X)$ is a primitive embedding and $i$ is a non-symplectic involution whose action on $H^{2}(X, \mathbb{Z})$ is conjugated to $\rho$.
 In fact, as a consequence of Theorem \ref{bcms}, the triples in $\mathcal{M}_{\rho}^{\langle 2\rangle}$ can be parametrized by the pairs $(X, D)$, where $D\in NS(X) $ is the ample divisor that
 generates $l(\langle 2\rangle) $ and $i$ is the anti-reflection in $D$. This result is expressed in \cite[Theorem 4.1]{BCMS}. Moreover, we have:

\begin{thm}\cite[Theorem 5.2]{BCMS}
Any two points $ (X, i_X,l_X),(Y, i_Y, l_Y )  \in \mathcal{M}_{\rho}^{\langle 2\rangle}$ are deformation equivalent.
\end{thm}

Although, any two points $ (X, i_X, l_X), (Y, i_Y,l_Y )  \in \mathcal{M}_{\rho}^{\langle 2\rangle}$ are deformation equivalent, we are interested in studying deformations such that every element of the family $\mathcal{X}$ is of the form $S^{[2]}$, where $S$ is a K3 surface. The following proposition is a combination of a result of P. Beri's thesis (\cite[Proposition 3.11]{pietro}) 
and Theorem \ref{Sarti}) and it will establish that. We denote by $K_{2t}$ the moduli space of $2t$-polarized K3 surfaces.



\begin{prop} \label{deformation}
Let $t\geq 2$ be a non-squared integer number. Assume that the equation $P_{t}(-1):x^{2}-ty^{2}=-1$ has a minimal positive solution $(a, b)$ and that $P_{4t}(5):x^{2}- 4ty^{2}=5$ has no solution.
Let $U \subset K_{2t}$ be the subset of points $(S,L) \in K_{2t}$  such that $b[L] - a\delta$ is ample. Then for any two pairs $(S_1, L_1)$ and $(S, L_2) \in U $ there exists a deformation family of hyperkähler manifolds $\pi \colon \mathcal{X} \to B$ over an analytic connected base $B$, a line bundle $\mathcal{L}$ on $\mathcal{X}$ and points $0, p \in B$ such that :
\begin{itemize}
\item $(\pi^{-1}(0), \mathcal{L}|_{\pi^{-1}(0)}) \simeq (S_1^{[2]}, b[L_1] -a \delta )$  and  $(\pi^{-1}(p), \mathcal{L}|_{\pi^{-1}(p)}) \simeq (S_2^{[2]}, b[L_2] -a \delta)$
\item for every $s \in B $, the pair $(\pi^{-1}(s)), \mathcal{L}|_{\pi^{-1}(s)}) \simeq (S^{[2]}, b[H] -a \delta)$ for some $(S,H) \in U$.
\end{itemize}

\end{prop}

\begin{rmk} \label{obs} Let the notation be as in Proposition \ref{deformation} and denote by $i_1$ the involution associated to the divisor $b[L_1] - a\delta$ and  $i_2$ the involution associated to the divisor $b[L_2] - a\delta$ as in Theorem \ref{bcms}. Then
$(S_1^{[2]}, i_1)$ and $(S_2^{[2]}, i_2)$ are deformation equivalent in the sense of Definition \ref{definv}, where every element of the deformation family is of the form $(S^{[2]}, i)$, where $S$ is a 2t-polarized K3 surface.

\end{rmk}

\section{Deformation of non-symplectic involutions}
For $S$ be an algebraic K3 surface such that $Pic(S) = \bZ L$ with $L^2 = 2t$ with $t$ satisfying certain numerical conditions, Theorem \ref{Sarti} predicts the existence of a non-symplectic involution $\sigma$ on the Hilbert square $S^{[2]}$. However, this Theorem describes $\sigma$ in terms of its action on $H^{2}(S, \mathbb{Z})$ without giving a precise geometric description of it.

In this section, we show that the set of surfaces $S_n$ in Example 1 constructed by Paiva and the author in \cite{paivaquedo} can be used to build geometric realizations for $\sigma$ for some values of $t$ using similar ideas as in \cite[Example 6.1]{BCMS} and \cite[Proposition 3.15]{pietro}. More precisely, we show that $(S^{[2]}, \sigma)$ and $(S_n^{[2]}, \kappa_1)$ are deformation equivalent, where $\kappa_1$ is a combination of two Beauville involutions, and every element of the deformation family is of the form $(S^{[2]}, i)$, where $S$ is a 2t-polarized K3 surface. But first, we will prove a proposition that will be helpful to us.

\begin{prop} \label{inw} Let $X$ be a hyperkähler manifold, $i \colon X \to X$ an involution and $f \colon X \to X$ an automorphism. Then, the invariant lattices $ H^{2}(X,\bZ)^{f}$ and  $H^{2}(X,\bZ)^{i\circ f \circ i}$ satisfy the following relation
	 $i^{*}(H^{2}(X,\bZ)^{f})=H^{2}(X,\bZ)^{i \circ f \circ i}$.
\end{prop}
\begin{proof}
We denote by $\kappa= i \circ f \circ i$. Given $W \in H^{2}(X,\bZ)^{f}$, we have $\kappa^{*} (i^{*}(W))=i^{*}f^{*}i^{*}(i^{*}(W))= i^{*}f^{*}(W)=i^{*}(W)$,
and so $i^{*}(W) \in H^{2}(X,\bZ)^{i\circ f \circ i} $.
Moreover, if $D \in H^{2}(X,\bZ)^{\kappa} $, we have $i^{*}(f^{*}(i^{*}(D))=D$. Then, since $i^{*}$ is an involution, we have that $i^{*-1}= i^{*}$, and so $f^{*}(i^{*}(D))=i^{*}(D)$, concluding our proof.
 \end{proof}

 We look at an example constructed by Paiva and the author in \cite{paivaquedo}:
 
 Given $n$ an integer number bigger than 1, we consider the bilinear form determined by the matrix
\begin{equation*}
 Q_n=\begin{pmatrix}
4 & 8n \\
8n & 2 
\end{pmatrix}.
\end{equation*}

By \cite[Corollary 2.9]{Morrison}, there is a K3 surface $S_n$ such that the group $\mathbb{Z}H + \mathbb{Z}W $ together with the above bilinear form corresponds to the Picard Lattice of $S_n$.

\begin{prop} The Hilbert square $S_n^{[2]}$ admits two Beauville involutions $i_1$ and $i_2$ whose invariant lattices are given by $H^{2}(S_n^{[2]}, \bZ)^{i_1}= \langle [H]- \delta \rangle$ and $H^{2}(S_n^{[2]}, \bZ)^{i_2}= \langle [8nW- H] - \delta \rangle.$

\end{prop}
\begin{proof}
Let $d_n= - disc (Q_n)=8(8n^{2}-1)$, by Lemma \ref{latticecondition}, we know that if the generalized Pell equation $P_{d_n}(-8):x^{2} - d_ny^2 = -8$ does not admit a solution, the surface $S_n$ does not admit -2-divisors, we see by Proposition \ref{pell2} that this condition is satisfied by the above matrix. Also, since $d_n$ is not a square, the only 0-divisor that $S_n$ admits is the trivial one. 
Therefore, the positive cone of the K3 surface coincides with the ample cone, so since $H$ and $8nW- H$ are divisors of square 4 and $\langle H,8nW-H \rangle= 64n^{2} - 4 >0$, we can suppose without loss of generality that they are both ample divisors of $S_n$. Moreover, using Theorem \ref{Mori}, we verify that they are very ample, and since $S_n$ does not admit -2-divisors, $S_n$ has no lines and we conclude that $[H]- \delta$ and $[8nW- H] - \delta$ on $NS(S_n^{[2]})$ generate the invariant lattice of Beauville involutions $i_1$ and $i_2$ on $S_n^{[2]}$, respectively.
\end{proof}

We define $\kappa_1:=i_2i_1i_2$ and $\kappa_2:=i_1i_2i_1$. Considering that $ \langle H, 8nW-H \rangle = 64n^{2} - 4= r+3$, we obtain the following corollary of Proposition \ref{inw}.

\begin{cor} \label{cor}
The invariant lattices $H^{2}(S_n^{[2]},\bZ)^{\kappa_1}$ and $H^{2}(S_n^{[2]},\bZ)^{\kappa_2}$ have rank 1 and they are generated, respectively, by the ample square 2  divisors $D_1=(64n^{2}- 5)H + 8n(64n^2 - 6)W - (64n^2-7) \delta$ and $D_2=(64n^{2}-5)H - 8nW - (64n^{2}- 7) \delta$.
\end{cor}
 \begin{proof}
 By Proposition \ref{inw}, $$H^{2}(S_n^{[2]},\bZ)^{\kappa_1}= i_2^{*}(H^2(X,\bZ)^{i_1})= \langle i_2^{*}([H]- \delta) \rangle=:\langle D_1 \rangle$$ and 
  $$H^{2}(S_n^{[2]},\bZ)^{\kappa_2}= i_1^{*}(H^2(X,\bZ)^{i_2})= \langle i_1^{*}([8nW- H] - \delta)\rangle=:\langle D_2 \rangle,$$
  which gives our first assertion.
 
 We denote $H$ by $d_1$  and $8nW- H$ by $d_2$ and recall that $i_1$ acts as the anti-reflection of $[H] - \delta$ on $H^{2}(S_n^{[2]}, \bZ)$, that is:

 \begin{align*}
 i^{*}_1 \colon H^2(S^{[2]}, \bZ) &\to H^2(S^{[2]}, \bZ)\\
v  &\mapsto \langle v, [d_1] - \delta \rangle ([d_1] - \delta) - v.
\end{align*}

 We compute $D_2=i^{*}_1([d_2] - \delta)$.
 
 \begin{align*}
    i_1^{*}([d_2] - \delta)&= \langle [d_2]- \delta, [d_1] - \delta \rangle([d_1] - \delta) - [d_2] + \delta \\
                &= ((r+3) -2)([d_1] - \delta) - [d_2] + \delta \\
                &= (r+1)[d_1] -[d_2] - r \delta.
\end{align*}
 
We conclude that $D_2=i_1^{*}([8nW- H] - \delta)=(64n^{2}-5)H - 8nW - (64n^{2}- 7) \delta$ and in analogous way, $D_1=i_2^{*}([H]- \delta)=(64n^{2}- 5)H + 8n(64n^2 - 6)W - (64n^2-7) \delta$. 

 \end{proof}
 
 \begin{thm}Consider $(S, L)$ a K3 surface with polarization $L$ of square $2t$, $t\geq 2$ and $Pic(S)= \bZ L$.  
 Then, 
 $S^{[2]}$ admits a unique non-symplectic involution $\sigma$ such that
 $(S^{[2]}, \sigma)$ can be deformed into $(S_n^{[2]}, \kappa_i)$, j=1,2, in a way that every element of the deformation family is of the form $(\Sigma^{[2]}, j) $, where  $\Sigma$ is $2t$-polarized K3 surface and $j$ an involution of $\Sigma^{[2]}$ if and only if $t$ can be written as $(64n^{2}-7)^{2}+1$ or  $2^{12} \cdot 5^{2}k^{4} - 2^{5} \cdot 7k^{2} +2$.
 \end{thm}
 \begin{proof} By Theorem \ref{Sarti}
and Remark \ref{obs1} , if  $t$ is not a square, the generalized Pell equation $P_{4t}(5):x^{2}- 4ty^{2}=5$ has no solution and the generalized Pell
equation $P_{t}(-1):x^{2}-ty^{2}=-1$ has a minimal solution $(a,b)$, then
$S^{[2]}$ admits a non-symplectic involution 
$\sigma$ that is associated with the divisor 
$D=b[L]- a \delta $ in the lines of Theorem \ref{bcms}. As a consequence of Theorem \ref{bcms} and Corollary \ref{cor}, we can associate $\kappa_1$ and $\kappa_2$ to $D_1$ and $D_2$, respectively.
By Remark \ref{obs},
we need to verify for which values of $t$, $(S^{[2]}, b[L] -a\delta)$, $(S_n^{[2]}, D_1)$ and 
$(S_n^{[2]}, D_2)$ are deformation equivalents as in Proposition \ref{deformation}.

For that, we have to find the possible minimal solutions $(a,b)$ of $P_t(-1)$ and polarizations $L_i$ on $S$, $L_i^{2}=2t$ that satisfy the following equality

 $$D_i= b [L_i] -a \delta $$

We notice that the only possible value for $a$ is $64n^{2} -7$. Also, it is easy to see that one of the possible solutions is obtained by setting $b=1$, 
and $L_1:=(64n^{2}- 5)H + 8n(64n^2 - 6)W$ and $L_2:=(64n^{2}-5)H - 8nW $. Then:

\begin{align*}  (L_i - (64n^{2} -7) \delta)^{2}  &= 2\\
                           L^{2}_i   -2 (64n^{2} -7)^{2}  &= 2\\
                            L^{2}_i &= 2( 1 +(64n^{2} -7)^{2}).
\end{align*}

Then, for $t=  1 +(64n^{2} -7)^{2}$, 
since we can verify that $P_{4t}(5)$  does not have a solution, we conclude that given an element $(S, L) \in K_{2t}$, $Pic(S)= \bZ H$, $S^{[2]}$ admits a non-symplectic involution $\sigma$ and that setting $b=1$ and $a=(64n^{2} -7)$, $(S^{[2]}, b[L]-a\delta)$ can be deformed into $(S^{[2]}_n, b[L_i]-a\delta)$ as in Proposition \ref{deformation}, implying that
 $(S^{[2]}, \sigma)$ can be deformed into $(S^{[2]}_n, \kappa_i)$ through a deformation as described in the statement.

 We now look for possible values for $b \neq 1$. We consider $$bL_1=(64n^{2}- 5)H + 8n(64n^2 - 6)W.$$

We have that  $L_1 \in Pic(S_n)$ if and only if $b$ is a common divisor of both $5-64n^{2}$ and
 $8n$, that it is, iff $g.c.d(5-64n^{2}, 8n) \neq 1$, or equivalently iff $n$ is multiple of $5$.
 In this case, $g.c.d(5-64n^{2}, 8n) =5$ and $b=5$. 
 These are also the necessary and sufficient conditions for
  $L_2$ to be an element of $ Pic(S_n)$ for $bL_2=(64n^{2}-5)H - 8nW $.

 Then, we have
 
 \begin{align*}
                        25L'^{2}_i&= 2( 1 +(64 \cdot 25k^{2} -7)^{2}) \\
                             L'^{2}_i&= 2( 2^{12} \cdot 5^{2}k^{4} - 2^{5} \cdot 7k^{2} +2)
 \end{align*}

Then in an analogous way as the previous case,  we obtain that for $t=  2^{12} \cdot 5^{2}k^{4} - 2^{5} \cdot 7k^{2} +2$, setting $b=5$ and $a=(64n^{2} -7)$,  
$(S^{[2]}, b[L]-a\delta)$ can be deformed into $(S^{[2]}_n, b[L'_i]-a\delta)$ as in Proposition \ref{deformation} and as we have seen these are the only possible cases for the deformation.

 \end{proof}

\section{Non-natural involution with rank 2 invariant lattice}
 
Given $X$ a hyperkähler fourfold of type $K3^{[2]}$ and a non-sympletic automorphism $\psi$, studying the invariant lattice $$ H^{2}(X, \bZ)^{\psi}= \{ D\in H^{2}(X, \bZ)  \text{ | } \psi^{*}D= D\} \subset H^{2}(X, \mathbb{Z})=U^{\oplus 3} \oplus E_{8}(-1)^{\oplus 2} \oplus \langle -2 \rangle $$ and its orthogonal complement $S_{\psi}$ tells us important information about $\psi$.
Bossière, Camere, and Sarti gave a classification of the lattices $ H^{2}(X, \bZ)^{\psi}$ and $S_{\psi}$ in \cite{BCS} for non-sympletic automorphism of order $p$ if $p=2,3$ or $7\leq p \leq 19 $.
In particular, for $p=2$, they proved that both $H^{2}(X, \bZ)^{\psi}$ and $S_{\psi}$ are indefinite 2-elementary lattices and  $H^{2}(X, \bZ)^{\psi}$ is hyperbolic.
Furthermore, they classified all non-isometric primitive embeddings of  $H^{2}(X, \bZ)^{\psi}$ in $U^{\oplus 3} \oplus E_{8}(-1)^{\oplus 2} \oplus \langle -2 \rangle$ (\cite[Proposition 8.2]{BCS})
and proved that each one of them can be realized as the invariant lattice of a non-sympletic involution of a hyperkähler manifold of type $K3^{[2]}$ (\cite[Theorem 8.5]{BCS}). However, although the existence of non-sympletic involutions is known, there are few geometrical examples of them. 

In the case when $H^{2}(X, \bZ)^{\psi}= \langle 2 \rangle \oplus \langle-2 \rangle$, there are two possible primitive embeddings (up to isometries) of $H^{2}(X, \bZ)^{\psi}$ in $U^{\oplus 3} \oplus E_{8}(-1)^{\oplus 2} \oplus \langle -2 \rangle$, and we can distinguish them by their orthogonal complement $S_{\psi}$. For the first case, we have that $S_{\psi}= U^{\oplus 2} \oplus E_8(-1) \oplus E_7(-1) \oplus \langle -2 \rangle ^{\oplus 2}$ and an example of such non-sympletic involution is due to Ohashi and Wandel. Their construction uses moduli space of K3 surfaces in \cite{OW}.
For the second case, $S_{\psi}= U^{\oplus 2} \oplus E_{8}(-1)^{\oplus 2} \oplus \langle -2 \rangle$ and a known example is a natural involution on the Hilbert square of a K3 surface (See \cite[Remark 8.4]{BCS}).

In this section, we construct a new example of a non-symplectic involution on the Hilbert square $S^{[2]}$ whose invariant lattice is $ \langle 2 \rangle \oplus\langle -2 \rangle$. The peculiarity of this particular construction, and what distinguishes it from the two cases above, is that our example has a clear geometrical description, whereas it is a non-natural involution.

For that, we look again at the family of surfaces $S_n$ constructed in the previous section, using the same notation as there. Since the family of matrices $Q_n$ has an ample divisor $W$ of square 2, \cite[Proposition 12]{paivaquedo}, we have that $W$ induces an involution $\varphi$ on $S_n$ such that $H^{2}(S_n, \bZ)^{\varphi}=\langle W \rangle$. We denote by $\varphi^{[2]}$ the induced natural involution on $S^{[2]}$. We consider $\iota=i_1\varphi^{[2]}i_1$, where $i_1$ is the Beauville involution defined in the previous section.


\begin{thm}
$\iota$ is a geometric example of a non-natural involution on $S_n^{[2]}$ with $H^{2}(S_n^{[2]}, \bZ)^{\iota}= \langle 2 \rangle \oplus\langle -2 \rangle$ and $S_{\iota}=  U^{\oplus 2} \oplus E_{8}(-1)^{\oplus 2} \oplus \langle -2 \rangle$.
\end{thm}
\label{ex2}
  
\begin{proof}
Recall that $$H^2(S_n^{[2]}, \bZ) \cong H^{2}(S_n,\bZ) \oplus \mathbb{Z} \delta.$$
As $\varphi^{[2]}$ is a natural involution on $ S_n^{[2]}$, $\varphi^{[2]*}(\delta)= \delta$, by Theorem \ref{BS}.
Moreover, $H^{2}(S_n, \bZ)^{\varphi^{*}}= \langle W\rangle$ by \cite[Proposition 12]{paivaquedo} and so $H^{2}(S_n^{[2]}, \bZ)^{\varphi^{[2]*}}= \langle W, \delta\rangle$.

Since $i_1^{*}$, acts as the anti-reflection of the divisor $[H] -\delta $, we can write $i_1^{*}$ in the base $ \{ H,W, \delta \}$ as
\begin{equation*}
i_1^{*}|_{NS(S_n^{[2]})}
=\begin{pmatrix}
3& 8n & 2\\
0 & -1 & 0 \\
-4 & -8n & -3
\end{pmatrix}.
\end{equation*}

Then, $i_1^{*}(W)= 8nH- W-8n \delta$ and  $i_1^{*}(\delta)= 2H -3\delta$.  
By Proposition \ref{inw}, $H^{2}(S^{[2]}, \bZ)^{\iota}= \langle 8nH- W-8n \delta,2H -3\delta   \rangle$


Moreover, since $i_1$ is an automorphism, we have:
\begin{itemize}
    \item $i_1^{*}(W)^{2}= W^{2}=2 $.
    \item  $i_1^{*}(\delta)^{2}= \delta^{2}=-2 $.
    \item $\langle i_1^{*}(W),i_1^{*}(\delta) \rangle = \langle \delta,W \rangle =0 $.
\end{itemize}

Finally, we denote by $$l_1 \colon \langle 2 \rangle \oplus \langle-2 \rangle \hookrightarrow NS(S_n^{[2]}) $$ the embedding such that $l_1(\langle2 \rangle \oplus \langle-2 \rangle ) =H^{2}(S^{[2]}, \bZ)^{\varphi^{[2]}}= \langle W, \delta\rangle$
and by $$l_2 \colon \langle 2 \rangle \oplus \langle-2 \rangle \hookrightarrow NS(S_n^{[2]}) $$ the embedding such that $l_2(\langle 2 \rangle \oplus \langle-2 \rangle ) =H^{2}(S^{[2]}, \bZ)^{\iota}= \langle 8nH- W-8n \delta, 2H -3\delta \rangle$.
We observe that they are isometric as a consequence of the commutativity of the following diagram.

\begin{center}
\begin{tikzcd}
\langle 2\rangle \oplus \langle -2 \rangle  \arrow[d, "id"] \arrow[r, "l_1", hook] & {NS(S_n^{[2]})} \arrow[d, "i_1^{*}"] \\
\langle 2\rangle \oplus \langle -2 \rangle  \arrow[r, "l_2", hook]                 & {NS(S_n^{[2]})}                     
\end{tikzcd}
\end{center}

By  \cite[Proposition 8.2, Example 8.6]{BCS},  we conclude that our example gives orthogonal complement $S_{\iota}=  U^{\oplus 2} \oplus E_{8}(-1)^{\oplus 2} \oplus \langle -2 \rangle$, as in the examples of natural involutions. 
\end{proof}

\bibliographystyle{plain}
\bibliography{refs}

\end{document}